\DeclareSymbolFont{AMSb}{U}{msb}{m}{n}
\documentclass[12pt,a4paper,leqno,noamsfonts]{amsart}
\usepackage[utf8]{inputenc}
\usepackage[T1]{fontenc} 
\usepackage[french]{babel}
\usepackage[dvipsnames]{xcolor}
\usepackage{graphicx,pifont,soul}
\usepackage[utopia]{mathdesign}
\usepackage[a4paper,top=3cm,bottom=3cm,
           left=3cm,right=3cm,marginparwidth=60pt]{geometry}
\usepackage{braket,caption,comment,mathtools,stmaryrd,enumitem}
\setlist[enumerate]{noitemsep}
\usepackage[usestackEOL]{stackengine}
\usepackage{multirow,booktabs,microtype,relsize}
\usepackage[colorlinks,bookmarks]{hyperref} %
      \hypersetup{colorlinks,%
            citecolor=britishracinggreen,%
            filecolor=black,%
            linkcolor=cobalt,%
            urlcolor=cornellred}
      \numberwithin{equation}{section}
\usepackage[french,english,capitalise]{cleveref}
\DeclareSymbolFont{usualmathcal}{OMS}{cmsy}{m}{n}
\DeclareSymbolFontAlphabet{\mathcal}{usualmathcal}

\usepackage[colorinlistoftodos]{todonotes}

\definecolor{cornellred}{rgb}{0.7, 0.11, 0.11}
\definecolor{britishracinggreen}{rgb}{0.0, 0.26, 0.15}
\definecolor{cobalt}{rgb}{0.0, 0.28, 0.67}

\newcommand{\BA}{{\mathbb{A}}}

\newcommand{\BC}{{\mathbb{C}}}

\newcommand{\BG}{{\mathbb{G}}}

\newcommand{\BN}{{\mathbb{N}}}

\newcommand{\BP}{{\mathbb{P}}}

\newcommand{\simto}{\,\widetilde{\to}\,}
\newcommand{\into}{\hookrightarrow}
\newcommand{\onto}{\twoheadrightarrow}

\DeclareMathOperator{\Hilb}{Hilb}

\DeclareMathOperator{\Quot}{Quot}

\DeclareMathOperator{\Var}{Var}

\DeclareMathOperator{\GL}{GL}

\DeclareMathOperator{\rk}{rg}

\DeclareMathOperator{\NCQuot}{ncQuot}

\DeclareFontFamily{OT1}{rsfs}{}
\DeclareFontShape{OT1}{rsfs}{n}{it}{<-> rsfs10}{}
\DeclareMathAlphabet{\curly}{OT1}{rsfs}{n}{it}

\newcommand\Hom{\operatorname{Hom}}
\newcommand\End{\operatorname{End}}
\newcommand{\lHom}{\mathscr{H}\kern-0.3em {o}\kern-0.2em{m}}
\newcommand{\RRlHom}{\mathbf{R}\kern-0.025em\mathscr{H}\kern-0.3em {o}\kern-0.2em{m}}
\newcommand{\lExt}{{\mathscr{E}\kern-0.2em xt}}
\newcommand{\RHom}{{\mathbf{R}\kern-0.07em\mathrm{Hom}}}

\newcommand\Supp{\operatorname{Supp}}

\newcommand{\OO}{\mathscr O}

\makeatletter
\newenvironment{proofof}[1]{\par
  \pushQED{\qed}%
  \normalfont \topsep6\p@\@plus6\p@\relax
  \trivlist
  \item[\hskip3\labelsep
        \itshape
    Démonstration du #1\@addpunct{.}]\ignorespaces
}{%
  \popQED\endtrivlist\@endpefalse
}
\makeatother

\usepackage[all]{xy}
\usepackage{tikz}
\usepackage{tikz-cd}
\usepackage{adjustbox}
\usepackage{rotating}
\usepackage{comment}

\usetikzlibrary{matrix,shapes,intersections,arrows,decorations.pathmorphing}
\tikzset{commutative diagrams/arrow style=math font}
\tikzset{commutative diagrams/.cd,
mysymbol/.style={start anchor=center,end anchor=center,draw=none}}
\newcommand\MySymb[2][\square]{%
  \arrow[mysymbol]{#2}[description]{#1}}
\tikzset{
shift up/.style={
to path={([yshift=#1]\tikztostart.east) -- ([yshift=#1]\tikztotarget.west) \tikztonodes}
}
}

\theoremstyle{definition}

\newtheorem*{lemma*}{Lemma}
\newtheorem*{theorem*}{Theorem}
\newtheorem*{example*}{Example}
\newtheorem*{fact*}{Fact}
\newtheorem*{notation*}{Notation}
\newtheorem*{definition*}{Definition}
\newtheorem*{prop*}{Proposition}
\newtheorem*{remark*}{Remarque}
\newtheorem*{corollary*}{Corollary}

\newtheorem*{conventions*}{Conventions}

\newtheorem{definition}{Definition}[section]

\newtheorem{remark}[definition]{Remarque}

\newtheoremstyle{thm} 
        {3mm}
        {3mm}
        {\slshape}
        {0mm}
        {\bfseries}
        {.}
        {1mm}
        {}
\theoremstyle{thm}

\newtheorem{lemma}[definition]{Lemme}

\newtheorem{thm}{Théorème}

\newcommand{\boldit}[1]{\boldsymbol{#1}} 
\newcommand{\nn}{\boldit{n}}
\newcommand{\expdim}{\mathrm{expdim}}

\title[Sur la lissité du schéma Quot ponctuel emboîté]{Sur la lissité du schéma Quot ponctuel emboîté}
\author{Sergej Monavari et Andrea T. Ricolfi}

\begin{document}

\maketitle

\begin{abstract}
Dans cet article on caractérise la lissité du \emph{schéma Quot ponctuel emboîté} d'une variété lisse  --- c'est-à-dire l'espace de modules paramétrant les drapeaux de quotients de dimension $0$ d'un faisceau localement libre fixé. Nos résultats étendent la classification de Cheah concernant les schémas de Hilbert ponctuels emboîtés.
\end{abstract}

\section{Introduction}
Soit $X$ une variété lisse et quasi-projective de dimension $m$, définie sur le corps $\BC$. Soit $E$ un faisceau localement libre de rang $r$ au dessus de $X$. Pour un entier fixé $d>0$ et un $d$-uplet $\nn = (0\leq n_1 \leq \cdots \leq n_d)$ d'entiers non-décroissants, on considère le \emph{schéma Quot ponctuel emboîté}
\[
\Quot_X(E,\nn) = \Set{\bigl[E \onto T_d \onto \cdots \onto T_1\bigr]|\dim(T_i) = 0,\,\chi(T_i) = n_i}
\]
où la dimension d'un faisceau cohérent $T$ est, par définition, la dimension de son support.

Dans cet article on donne des conditions nécessaires et suffisantes pour que le schéma $\Quot_X(E,\nn)$ soit lisse. Quand $d=1$ on retrouve le schéma Quot de Grothendieck et par abus on remplace l'écriture $\nn = (0\leq n)$ par l'entier $n \in \BN$ correspondant. Sans que cela impacte la généralité de notre propos, on suppose au cours du théorème suivant que $\nn$ est de la forme $\nn = (0<n_1<\cdots<n_d)$.

\begin{thm}\label{main_thm_INTRO}
Soit $(X,E,\nn)$ comme ci-dessus. Alors $\Quot_X(E,\nn)$ est lisse dans les cas suivants:
\begin{enumerate}
    \item Si $m=1$, pour tout choix de $(E,d,\nn)$, \label{previous_paper}
     \item si $d=1$  et $n=1$, \label{projective_bundle}
     \item si $r=1$, dans les cas suivants:
     \begin{enumerate}
         \item  $m=2,d=1$, pour tout choix de $n$, \label{Fogarty}
    \item  $m=d=2$ et $\nn=(n,n+1)$, \label{surfaces_rank_1}
    \item  $m\geq 3,d=1$ et $n\leq 3$, \label{higherdim_1}
    \item  $m\geq 3,d=2$ et $\nn=(1,2),(2,3)$, \label{higherdim_2}
     \end{enumerate}
\end{enumerate}
Dans tous les autres cas, $\Quot_X(E,\nn)$ est singulier.
\end{thm}

On va démontrer le \Cref{main_thm_INTRO} de la façon suivante: on se ramène d'abord au cas $(X,E) = (\BA^m,\OO^{\oplus r})$, on généralise ensuite la classification de Cheah \cite{MR1616606} pour $r=1$ (listant tous les schémas de Hilbert ponctuels emboîtés lisses) au rang $r$ arbitraire; enfin on exclut toutes les exceptions à priori possibles, en produisant explicitement des points singuliers.

\smallbreak
On remarque ici que dans le cas $d=r=1$, correspondant au \emph{schéma de Hilbert de $n$ points} $\Hilb^n(X)$, il est connu que la lissité s'obtient si et seulement si $m\leq 2$ ou bien $n\leq 3$. Si $r>1$, le schéma Quot de Grothendieck $\Quot_X(\OO^{\oplus r},n)$ est lisse si $X$ est une courbe lisse, par contre il est singulier (mais irréductible, de dimension $n(r+1)$, voir \cite{Ellingsrud_Lehn} et \cite[Example 3.3]{cazzaniga2020framed}) si $X$ est une surface.

La cohomologie de $\Quot_X(E,\nn)$ a été étudiée en détail par Mochizuki  \cite{Mochizuki_FILT} lorsque $X$ est une courbe lisse; dans ce cas-là, le motif $[\Quot_X(E,\nn)] \in K_0(\Var_{\BC})$ de ce schéma a été calculé explicitement dans notre article \cite{bestpaperever}.

\section{Propriétés de l'espace de modules}

On fixe, avec les notations précédentes, un triplet $(X,E,\nn)$ formé d'un faisceau localement libre $E$ au dessus d'une variété lisse $X$, et un $d$-uplet d'entiers $\nn = (0\leq n_1\leq \cdots\leq n_d)$ pour un entier $d>0$. On rappelle que l'on utilise la notation $m=\dim X$ et $r=\rk E$. On remarque aussi que, si $n_d=1$, le schéma $\Quot_X(E,\nn)$ est isomorphe à $\BP(E)$, et notamment est lisse de dimension $m+r-1$. Ce fait sera exploité dans la \Cref{sec_expdim}.

\subsection{Espace tangent}\label{sec_tangentspace}
Comme démontré en \cite[Prop.~2.1]{bestpaperever}, on peut décrire l'espace tangent du schéma $\Quot_X(E,\nn)$ en un point $z = [E \onto T_d \onto\cdots \onto T_1]$ comme le noyeau d'une application $\BC$-linéaire appropriée,
\[
T_z\Quot_X(E, \nn)=\ker\left(\bigoplus_{i=1}^d\Hom(K_i, T_i)\xrightarrow{\Delta_z} \bigoplus_{i=1}^{d-1}\Hom(K_{i+1}, T_{i})\right),
\]
où l'on pose $K_i = \ker (E \onto T_i)$. On omet la définition précise de $\Delta_z$. On n'en fera pas usage dans nos preuves (le lecteur pourra en trouver une définition dans \cite[Section 2]{bestpaperever} ou encore, dans une forme équivalente, dans \cite{Mochizuki_FILT}).

\subsection{Le morphisme somme directe}
Supposons que l'on ait une décomposition $\nn = \nn_1+\cdots+\nn_s$, où tous les $\nn_k = (n_{k1}\leq \cdots\leq n_{kd})$ sont des suites non-décroissantes d'entiers non-négatifs <<plus petites>> que $\nn$. La notation <<somme>> ci-dessus signifie bien-sûr que $n_i = \sum_{1\leq k\leq s} n_{ki}$ pour tout $i=1,\ldots,d$. Considérons l'ouvert
\[
U \into \prod_{1\leq k\leq s} \Quot_X(E,\nn_k)
\]
paramétrant les $s$-uplets de quotients emboîtés
\[
z_k = \bigl[ E \onto T_{kd} \onto \cdots \onto T_{k1} \bigr] \in \Quot_X(E,\nn_k),\qquad k=1,\ldots,s,
\]
tels que le support de $T_{kd}$ soit disjoint du support de $T_{ld}$ pour tout $1\leq k\neq l\leq s$. Alors on a un morphisme de schémas
\[
\begin{tikzcd}
U \arrow{r}{\oplus} & \Quot_X(E,\nn)
\end{tikzcd}
\]
qui associe à un $s$-uplet $(z_1,\ldots,z_s)$ le point
\[
\bigl[E \onto T_{1d}\oplus \cdots \oplus T_{sd} \onto \cdots \onto T_{11}\oplus \cdots \oplus T_{s1} \bigr] \in \Quot_X(E,\nn).
\]
Une application immédiate du critère infinitésimal montre que ce morphisme est étale.

\subsection{Dimension attendue}\label{sec_expdim}
Fixons $\nn=(n_1\leq \dots \leq n_d)$ et une décomposition
$\nn = \sum_{k=1}^{n_d} \nn_k,$
où tout $\nn_k = (n_{k1}\leq \cdots \leq n_{kd})$ satisfait à la condition $n_{kd} = 1$.  Dans le produit
\[
\prod_{k=1}^{n_d} \Quot_X(E,\nn_k) \cong \BP(E)^{n_d}
\]
on considère le sous-schéma ouvert $U_{\nn}$ paramétrant les $n_d$-uplets de quotients dont les supports sont deux à deux disjoints. L'ouvert $U_{\nn}$ est lisse de dimension $n_d(m+r-1)$. Comme $U_{\nn}$ est \'etale au dessus de $\Quot_X(E,\nn)$, à travers le morphisme somme directe, on peut définir la \emph{dimension attendue}
\[
\expdim \Quot_X(E,\nn) = n_d(m+r-1).
\]
En effet, $\Quot_X(E,\nn)$ contient un ouvert lisse (l'image de $U_{\nn}$) de cette dimension. Dans le cas du schéma de Hilbert de $n$ points $\Hilb^n(X)$, l'image de $U_n$ paramètre les $n$-uplets de points distincts (à permutation près). Sa dimension est bien $n\cdot \dim(X)$. Ce nombre est la dimension de $\Hilb^n(X)$ lorsqu'il est irréductible, car la clôture de Zariski de cet ouvert-là, que l'on appelle la \emph{smoothable component}, est toujours une composante irréductible.

\subsection{Connexion}\label{sec_connectedness}
Si $X$ est irréductible, le schéma $\Quot_X(E,\nn)$ est connexe \cite[Thm.~1.4]{bestpaperever}. Alors, si l'on trouve un point $z \in \Quot_X(E,\nn)$ tel que 
\[
\dim_{\BC} T_z \Quot_X(E,\nn) > \expdim \Quot_X(E,\nn) = n_d (m+r-1),
\]
il en résulte que $z$ est forcément un point singulier de $\Quot_X(E,\nn)$.  

\section{Démonstration du théorème}
Nous allons réduire notre analyse sur l'existence des singularités concernant un couple <<global>> $(X,E)$ à une analyse concernant un couple <<local>> $(\BA^m,\OO^{\oplus r})$.

\begin{lemma}\label{lem_ etale chart}
Soit $X$ une variété lisse et quasi-projective de dimension $m$ sur $\BC$, et soit $E$ un faisceau localement libre de rang $r$ au dessus de $X$. Alors $\Quot_X(E,\nn)$ est lisse si et seulement si $\Quot_{\BA^m}(\OO^{\oplus r},\nn)$ est lisse.
\end{lemma}

\begin{proof}
L'énoncé résulte du fait que $\Quot_X(E,\nn)$ est localement une carte étale pour $\Quot_{\BA^m}(\OO^{\oplus r},\nn)$. On détaille ce fait dans la suite.

Considérons d'abord le cas $d=1$. Soit $U\subset X$ une sous-variété ouverte tel que $E|_U = \OO_U^{\oplus r}$. Supposons que l'on ait un morphisme étale $\varphi\colon U \to \BA^m$. Si l'on écrit $V^{\varphi}_{r,n}$ pour le sous-schéma ouvert de $\Quot_U(\OO_U^{\oplus r},n)$ paramétrant les quotients $[\OO_U^{\oplus r} \onto T]$ tels que $\varphi|_{\Supp(T)}$ soit injectif, on peut bien définir un morphisme {\'etale}
\cite[Prop.~A.3]{BR18}
\[
\Phi_{n}\colon V_{r,n}^\varphi \to \Quot_{\BA^m}(\OO_{\BA^m}^{\oplus r},n)
\]
en associant $[\OO^{\oplus r} \onto T] \mapsto [E \to \varphi_\ast \varphi^\ast E = \varphi_\ast \OO^{\oplus r} \onto \varphi_\ast T]$. En variant $(U,\varphi\colon U \to \BA^m)$ pour couvrir $\BA^m$ tout entier, on peut facilement confirmer le résultat dans le cas $d=1$.

Pour le cas général, fixons $\nn = (0<n_1\leq\cdots\leq n_d)$ et $(U,\varphi)$ comme ci-dessus.
Le produit des morphismes \'etales $\Phi_{n_i}$ nous donne un morphisme \'etale $\Phi_{\nn}$ qui apparaît dans un diagramme 
\[
\begin{tikzcd}[row sep=large,column sep=large]
Z_{\nn}^{\varphi}\MySymb{dr} \arrow[hook]{r} \arrow[swap]{d}{\textrm{\'etale}} & \displaystyle\prod_{1\leq i\leq d} V_{r,n_i}^{\varphi} \arrow{d}{\Phi_{\nn}} \\
\Quot_{\BA^m}(\OO_{\BA^m}^{\oplus r},\nn) \arrow[hook]{r} & \displaystyle\prod_{1\leq i\leq d} \Quot_{\BA^m}(\OO_{\BA^m}^{\oplus r},n_i)
\end{tikzcd}
\]
où les flèches horizontales sont des immersions fermées.

On peut facilement verifier que $Z_{\nn}^{\varphi}$ est aussi l'intersection schématique
\[
\begin{tikzcd}[row sep=large,column sep=large]
Z_{\nn}^{\varphi}\MySymb{dr} \arrow[hook]{r} \arrow[hook]{d}
& \displaystyle\prod_{1\leq i\leq d} V_{r,n_i}^{\varphi} \arrow[hook]{d}{\textrm{ouvert}} \\
\Quot_{U}(\OO_U^{\oplus r},\nn) \arrow[hook]{r}{\textrm{fermé}} & \displaystyle\prod_{1\leq i\leq d} \Quot_{U}(\OO_U^{\oplus r},n_i)
\end{tikzcd}
\]
dans un produit de schémas Quot classiques; comme $\Quot_{U}(\OO_U^{\oplus r},\nn)\subset \Quot_{X}(E,\nn)$ est ouvert, on a trouvé un sous-schéma ouvert $Z_{\nn}^{\varphi} \subset \Quot_{X}(E,\nn)$ qui admet un morphisme étale vers $\Quot_{\BA^m}(\OO^{\oplus r}_{\BA^m},\nn)$. En faisant varier $(U,\varphi\colon U \to \BA^m)$ tout comme dans le cas $d=1$ on obtient le résultat.
\end{proof}

On aborde désormais la démonstration de notre résultat principal.

\begin{proofof}{\Cref{main_thm_INTRO}}
Grâce au Lemme \ref{lem_ etale chart} on peut supposer que $(X,E)=(\BA^m, \OO^{\oplus r}_{\BA^{m}})$. La lissité dans le cas $m=1$, voir~\eqref{previous_paper}, est démontrée dans notre article \cite[Prop.~2.1]{bestpaperever} et dans \cite[Prop.~2.1]{Mochizuki_FILT}. La lissité  dans les cas \eqref{Fogarty}--\eqref{higherdim_2} a été démontrée par Cheah \cite[Theorem, p.~43]{MR1616606}. Enfin, \eqref{projective_bundle} découle de l'isomorphisme $\Quot_{\BA^m}(\OO^{\oplus r}, 1)\cong \BA^{m}\times \BP^{r-1}$ (voir aussi Remarque \ref{rem: noncomm}). Il reste à prouver qu'il n'existe pas d'autres schémas Quot ponctuels emboîtés lisses.

On note que si $\Hilb^{\nn}(\BA^m) = \Quot_{\BA^m}(\OO,\nn)$ est singulier, alors il en est de même de $\Quot_{\BA^m}(\OO^{\oplus r},\nn)$ pour tout $r>1$. En effet, le tore $\BG_m^r$ opère canoniquement sur $\Quot_{\BA^m}(\OO^{\oplus r},\nn)$, et $\Hilb^{\nn}(\BA^m)$ est une composante connexe du sous-schéma des points fixes \cite[Prop.~3.1]{bestpaperever}.

Comme Cheah a démontré que $\Hilb^{\nn}(\BA^m)$ est singulier chaque fois qu'il ne tombe pas dans les cas \eqref{previous_paper},\eqref{projective_bundle},\eqref{Fogarty}--\eqref{higherdim_2}, on déduit que, si $r>1$, le schéma $\Quot_{\BA^m}(\OO^{\oplus r},\nn)$ est singulier dans les cas suivants:
\begin{enumerate}
\item si $d\geq 3$, pour tout choix de $\nn$,
\item si $m=2$, $d=2$, $\nn=(n, n')$ et $n'-n\geq 2$,
\item si $m\geq 3$, $d=1$, $n\geq 4$,
\item si $m\geq 3$, $d=2$, $\nn\neq (1,2),(2,3)$.
     \end{enumerate}
Il ne reste plus qu'à démontrer que $\Quot_{\BA^m}(\OO^{\oplus r},\nn)$ est singulier dans les cas suivants:
\begin{enumerate}
    \item [(A)] si $m\geq 2$, $r\geq 2$, $d=1$ et $n\geq 2$,\label{first case}
    \item [(B)] si $m\geq 2$, $r\geq 2$, $d=2$ et $\nn=(n,n+1)$.
    \label{last_case}
\end{enumerate}
Le cas (A) (resp.~(B)) est l'énoncé du Lemme \ref{lem_ classical sing Quot} (resp.~Lemme \ref{lem_ quot singular two-nested}).
\end{proofof}

\begin{remark}\label{rem: noncomm}
Soit $E$ un faisceau cohérent au dessus d'une variété $X$. L'isomorphisme $\Quot_X(E,1) = \BP(E)$ s'obtient en comparant les foncteurs de modules. En revanche, le cas $(X,E) = (\BA^m,\OO^{\oplus r})$, qui entraîne $\Quot_{\BA^m}(\OO^{\oplus r},1) = \BA^m \times \BP^{r-1}$, s'obtient également à travers une présentation explicite du schéma $\Quot_{\BA^m}(\OO^{\oplus r},n)$ en tant que sous-schéma fermé du \emph{schéma Quot non-commutatif}
\[
\NCQuot_m^{n,r} = 
\Set{
(A_1,\ldots,A_m,v_1,\ldots,v_r) \in \End_{\BC}(\BC^n)^m \times (\BC^n)^r |
  \begin{array}{c} (v_1,\ldots,v_r)\textrm{ est } \\
  (A_1,\ldots,A_m)\textrm{-stable}
  \end{array}}\Bigg{/}\GL_n,
\]
où $\GL_n$ opère par conjugaison sur les endomorphismes et par multiplication à gauche sur les vecteurs; enfin, la condition de stabilité se lit de la façon suivante: le sous-espace de $\BC^n$ engendré par les vecteurs obtenus en appliquant tous les monômes possibles en $A_1,\ldots,A_m$ au vecteurs $v_1,\ldots,v_r$ co\"{i}ncide avec $\BC^n$ tout entier. On voit facilement que la variété $\NCQuot_m^{n,r}$ est lisse de dimension $(m-1)n^2+rn$. Au cas où $n=1$, l'immersion (qui dans le cas général est définie par les relations $[A_i,A_j] = 0$) est triviale, et l'action de $\GL_1$ est aussi triviale sauf sur les $r$-uplets de nombres complexes $(v_1,\ldots,v_r) \in \BC^r$, qui ne peuvent pas être tous $0$ grâce à la condition de stabilité. Ceci fournit une démonstration directe de la décomposition $\Quot_{\BA^m}(\OO^{\oplus r},1) = \BA^m \times \BP^{r-1}$.
\end{remark}

Pour compléter la démonstration du \Cref{main_thm_INTRO} il nous reste à traiter les cas (A) et (B).

\begin{lemma}\label{lem_ classical sing Quot}
Soit $m\geq 2, r\geq 2,n\geq 2$. Alors $\Quot_{\BA^m}(\OO^{\oplus r}, n)$ est singulier.
\end{lemma}

\begin{proof}
Nous commençons par démontrer l'énoncé dans le cas $n=2$.

Considérons un point $z \in \Quot_{\BA^m}(\OO^{\oplus r},2)$ représenté par une suite exacte
\[
0 \to \mathfrak m_0^{\oplus 2} \oplus \OO^{\oplus r - 2} \to \OO^{\oplus r} \to \OO_0^{\oplus 2} \to 0
\]
où $\mathfrak m_0 = (x_1,\ldots,x_m) \subset \OO = \BC[x_1,\ldots,x_m]$ est l'idéal de l'origine $0 \in \BA^m$ et $\OO_0 = \OO/\mathfrak m_0$ est son faisceau structural. On obtient
\begin{align*}
    \dim_{\BC} T_z\Quot_{\BA^m}(\OO^{\oplus r},2) 
    &= \dim_{\BC}\Hom_{\OO}(\mathfrak m_0^{\oplus 2} \oplus \OO^{\oplus r - 2},\OO_0^{\oplus 2}) \\
    &=\dim_{\BC}\Hom_{\OO}(\mathfrak m_0^{\oplus 2},\OO_0^{\oplus 2}) + \dim_{\BC} \Hom_{\OO}(\OO^{\oplus r - 2},\OO_0)^{\oplus 2} \\
    &=4 m + 2 (r-2),
\end{align*}
qui est plus grand que $\expdim \Quot_{\BA^m}(\OO^{\oplus r},2) = 2(m+r-1)$ comme $m\geq 2$. En exploitant la connexion du schéma Quot (voir la \Cref{sec_connectedness}), le calcul ci-dessus montre que $z$ est bien un point singulier.

On suppose désormais que $n\geq 3$.
Considérons le sous-schéma ouvert
\[
U \into \Quot_{\BA^m}(\OO^{\oplus r},2) \times \Quot_{\BA^m}(\OO^{\oplus r},1)^{n-2}
\]
paramétrant les $(n-1)$-uplets de quotients dont les supports sont deux à deux disjoints. Choisissons un point $u \in U$ de la forme $u = (\mathfrak m_0^{\oplus 2} \oplus \OO^{\oplus r - 2},\mathfrak m_{p_1}\oplus \OO^{\oplus r-1},\ldots,\mathfrak m_{p_{n-2}}\oplus \OO^{\oplus r-1})$, où $0\neq p_i \in \BA^m$ pour tout $1\leq i\leq n-2$ et $p_i\neq p_j$ pour $1\leq i\neq j\leq n-2$. Le schéma $U$ est étale au dessus de $\Quot_{\BA^m}(\OO^{\oplus r},n)$ par le morphisme somme directe. On note $v$ l'image du point $u$ par ce morphisme. On trouve
\begin{align*}
    \dim_{\BC} T_v \Quot_{\BA^m}(\OO^{\oplus r},n) 
    &= \dim_{\BC} T_u U \\
    &= 4 m + 2 (r-2) + (n-2) (m+r-1) \\
    &= n(m+r-1) + 2m-2,
\end{align*}
qui est plus grand que $\expdim \Quot_{\BA^m}(\OO^{\oplus r},n) = n(m+r-1)$ comme $m\geq 2$. Encore une fois grâce à la connexion du schéma Quot, ceci prouve le résultat.
\end{proof}

\begin{lemma}\label{lem_ quot singular two-nested}
Soit $m\geq 2, r\geq 2$ et $\nn=(n,n+1)$ pour $n\geq 1$. Alors $\Quot_{\BA^m}(\OO^{\oplus r}, \nn)$ est singulier.
\end{lemma}

\begin{proof}
On commence par montrer l'énoncé dans le cas $n=1$.

Considérons un point $z \in \Quot_{\BA^m}(\OO^{\oplus r},(1,2))$ representé par les quotients emboîtés
\[
 \bigl[ \OO^{\oplus r}\onto \OO_0^{\oplus 2} \onto \OO_0 \bigr],
\]
et écrivons encore une fois $\mathfrak m_0 = (x_1,\ldots,x_m) \subset \OO = \BC[x_1,\ldots,x_m]$ pour l'idéal de l'origine $0 \in \BA^m$. Comme on l'a rappelé à la \Cref{sec_tangentspace}, l'espace tangent en $z$ est donné par:
\begin{multline*}
T_z\Quot_{\BA^m}(\OO^{\oplus r},(1,2))=\\\ker\left(\Hom_\OO(\mathfrak{m}_0\oplus \OO^{\oplus r-1}, \OO_0)\oplus \Hom_\OO(\mathfrak{m}_0^{\oplus 2}\oplus \OO^{\oplus r-2}, \OO_0^{\oplus 2}) \xrightarrow{\Delta_z}  \Hom_\OO(\mathfrak{m}_0^{\oplus 2}\oplus \OO^{\oplus r-2}, \OO_0) \right).
\end{multline*}
D'autre part, les espaces vectoriels apparaissant en $\Delta_z$ satisfont
\begin{align*}
    \dim_{\BC} \Hom_\OO(\mathfrak{m}_0\oplus \OO^{\oplus r-1}, \OO_0) &= m+r-1\\
    \dim_{\BC} \Hom_\OO(\mathfrak{m}_0^{\oplus 2}\oplus \OO^{\oplus r-2}, \OO_0^{\oplus 2}) &=4m+2(r-2) \\
    \dim_{\BC} \Hom_\OO(\mathfrak{m}_0^{\oplus 2}\oplus \OO^{\oplus r-2}, \OO_0)&= 2m+r-2.
\end{align*}
On obtient alors
\begin{align*}
    \dim_{\BC} T_z\Quot_{\BA^m}(\OO^{\oplus r},(1,2))&\geq (m+r-1)+(4m+2(r-2))- (2m+r-2)\\
    &> 2(m+r-1) = \expdim \Quot_{\BA^m}(\OO^{\oplus r},(1,2)).
\end{align*}
Ceci entraîne que $z$ est un point singulier par notre remarque à la \Cref{sec_connectedness}.

On va maintenant supposer que $n\geq 2$.
Considérons le sous-schéma ouvert
\[
U \into \Quot_{\BA^m}(\OO^{\oplus r},(1,2)) \times \Quot_{\BA^m}(\OO^{\oplus r},(1,1))^{n-1}
\]
paramétrant les $n$-uplets de quotients dont les supports sont deux à deux disjoints. Choisissons un point $u = (z,z_1,\ldots,z_{n-1}) \in U$, où $z$ est comme ci-dessus et $z_i$ est representé par des quotients emboîtés
\[
\bigl[\OO^{\oplus r} \onto \OO_{p_i} \simto \OO_{p_i}\bigr],\qquad p_i \in \BA^m.
\]
On va supposer également que $0\neq p_i \in \BA^m$ pour tout $i$ et que $p_i\neq p_j$ pour $1\leq i\neq j\leq n-1$. Le schéma $U$ est étale au dessus de $\Quot_{\BA^m}(\OO^{\oplus r},(n,n+1))$, par le morphisme somme directe. On note $v$ l'image du point $u$ par ce morphisme.
On trouve
\begin{align*}
    \dim_{\BC} T_v \Quot_{\BA^m}(\OO^{\oplus r},(n,n+1)) 
    &= \dim_{\BC} T_uU \\
    &>2(m+r-1) + (n-1)(m+r-1) \\
    &=(n+1)(m+r-1) \\
    &=\expdim \Quot_{\BA^m}(\OO^{\oplus r},(n,n+1)).    
\end{align*}
Le point $v$ est donc un point singulier.
\end{proof}

\subsection*{Remerciements}
S.M.~est financé par NWO grant TOP2.17.004. A.R.~est financé par Dipartimenti di Eccellenza.

\bibliographystyle{amsplain-nodash} 
\bibliography{bib}

\ifx\undefined\bysame
\newcommand{\bysame}{\leavevmode\hbox to3em{\hrulefill}\,}
\fi
\begin{thebibliography}{1}

\bibitem{BR18}
Sjoerd Beentjes and Andrea~T. Ricolfi, {\em {Virtual counts on Quot schemes and
  the higher rank local DT/PT correspondence}}, To appear in Math. Res. Lett.,
  2018.

\bibitem{cazzaniga2020framed}
Alberto Cazzaniga and Andrea~T. Ricolfi, {\em {Framed sheaves on projective
  space and Quot schemes}}, Math. Z.
  {\href{https://link.springer.com/article/10.1007/s00209-021-02802-x?wt_mc=Internal.Event.1.SEM.ArticleAuthorOnlineFirst&utm_source=ArticleAuthorOnlineFirst&utm_medium=email&utm_content=AA_en_06082018&ArticleAuthorOnlineFirst_20210629}{Online}}
  (2021).

\bibitem{MR1616606}
Jan Cheah, {\em Cellular decompositions for nested {H}ilbert schemes of
  points}, Pacific J. Math. {\bf 183} (1998), no.~1, 39--90.

\bibitem{Ellingsrud_Lehn}
Geir Ellingsrud and Manfred Lehn, {\em Irreducibility of the punctual quotient
  scheme of a surface}, Ark. Mat. {\bf 37} (1999), no.~2, 245--254.

\bibitem{Mochizuki_FILT}
Takuro Mochizuki, {\em {The structure of the cohomology ring of the filt
  schemes}}, {\href{https://arxiv.org/abs/math/0301184}{ArXiv: 0301184}}, 2003.

\bibitem{bestpaperever}
Sergej Monavari and Andrea~T. Ricolfi, {\em {On the motive of the nested Quot
  scheme of points on a curve}},
  \href{https://arxiv.org/abs/2106.11817}{ArXiv:2106.11817}, 2021.

\end{thebibliography}

\bigskip

\bigskip
\noindent
\emph{Sergej Monavari}, \texttt{s.monavari@uu.nl} \\
\textsc{Mathematical Institute, Utrecht University, P.O.~Box 80010 3508 TA Utrecht, The Netherlands}

\bigskip
\noindent
\emph{Andrea T.~Ricolfi}, \texttt{aricolfi@sissa.it} \\
\textsc{Scuola Internazionale Superiore di Studi Avanzati (SISSA), Via Bonomea 265, 34136 Trieste, Italy}

\end{document}